\documentclass[12pt,twoside]{amsart}

\usepackage[latin1]{inputenc}
\usepackage[OT1]{fontenc}
\usepackage{amsmath}
\usepackage{amsthm}
\usepackage{amssymb}
\usepackage[all]{xy}
\usepackage{hyperref}

\newtheorem{thm}{Theorem}[section]

\newtheorem{proposition}[thm]{Proposition}

\newtheorem{lemma}[thm]{Lemma}
\newtheorem{corollary}[thm]{Corollary}

\numberwithin{equation}{section}

\theoremstyle{definition}
\newtheorem{definition}[thm]{Definition}
\newtheorem{remark}[thm]{Remark}
\newtheorem{example}[thm]{Example}

\newcommand{\qqed}{\hspace*{\fill}$\Box$}
\newcommand{\supp}{\operatorname{Supp}}

\newcommand{\Db}{{\rm D}^{\rm b}}
\newcommand{\D}{{\rm D}}

\newcommand{\Aut}{{\rm Aut}}

\newcommand{\Pic}{{\rm Pic}}

\newcommand{\Hom}{{\rm Hom}}
\newcommand{\Spec}{{\rm Spec}}

\newcommand{\Stab}{{\rm Stab}}

\newcommand{\id}{{\rm id}}

\newcommand{\Ext}{{\rm Ext}}
\newcommand{\ka}{{\mathcal A}}
\newcommand{\kb}{{\mathcal B}}
\newcommand{\kc}{{\mathcal C}}

\newcommand{\ke}{{\mathcal E}}

\newcommand{\kh}{{\mathcal H}}
\newcommand{\ki}{{\mathcal I}}

\newcommand{\kl}{{\mathcal L}}

\newcommand{\km}{{\mathcal M}}
\newcommand{\ko}{{\mathcal O}}
\newcommand{\kp}{{\mathcal P}}

\newcommand{\ks}{{\mathcal S}}
\newcommand{\kt}{{\mathcal T}}

\newcommand{\kz}{{\mathcal Z}}

\newcommand{\IC}{\mathbb{C}}

\newcommand{\IP}{\mathbb{P}}

\newcommand{\IZ}{\mathbb{Z}}

\DeclareMathOperator{\cone}{cone}

\DeclareMathOperator{\Coh}{\bf{Coh}}

\DeclareMathOperator{\FM}{\mathrm{FM}}
\DeclareMathOperator{\HH}{HH}
\DeclareMathOperator{\rk}{rk}

\renewcommand{\to}{\xymatrix@1@=15pt{\ar[r]&}}
\renewcommand{\rightarrow}{\xymatrix@1@=15pt{\ar[r]&}}
\renewcommand{\mapsto}{\xymatrix@1@=15pt{\ar@{|->}[r]&}}
\renewcommand{\twoheadrightarrow}{\xymatrix@1@=15pt{\ar@{->>}[r]&}}
\renewcommand{\hookrightarrow}{\xymatrix@1@=15pt{\ar@{^(->}[r]&}}
\newcommand{\congpf}{\xymatrix@1@=15pt{\ar[r]^-\sim&}}
\renewcommand{\cong}{\simeq}

\newcommand{\G}{{\mathcal G}}
\DeclareMathOperator{\Ho}{H}
\newcommand{\sym}{\mathfrak S}
\newcommand{\reg}{\mathcal O}
\DeclareMathOperator{\Inf}{Inf}
\DeclareMathOperator{\triv}{triv}

\DeclareMathOperator{\sgn}{sgn}
\DeclareMathOperator{\Res}{Res}
\DeclareMathOperator{\nd}{nd}

\newcommand{\alt}{\mathfrak a}
\newcommand{\lex}{<_{\mathrm{lex}}}

\begin{document}

\title[On the derived category...]{On the derived category of the Hilbert scheme of points on an Enriques surface}
%\title[Spherical functors via Hilbert schemes...]{Semi-orthogonal decompositions on Hilbert schemes of points on Enriques surfaces and induced autoequivalences}

\author[A.\ Krug]{Andreas Krug}
\address{Mathematisches Institut, Universit\"at Bonn\\ Endenicher Allee 60\\ 53115 Bonn, Germany}
\email{akrug@math.uni-bonn.de}

\author[P.\ Sosna]{Pawel Sosna}
\address{Fachbereich Mathematik der Universit\"at Hamburg\\
Bundesstra\ss e 55\\
20146 Hamburg, Germany}
\email{pawel.sosna@math.uni-hamburg.de}

\begin{abstract}
We use semi-orthogonal decompositions to construct autoequivalences of Hilbert schemes of points on Enriques surfaces and of Calabi--Yau varieties which cover them. While doing this, we show that the derived category of a surface whose irregularity and geometric genus vanish embeds into the derived category of its Hilbert scheme of points.
\end{abstract}
\maketitle

\section{Introduction}

The bounded derived category of coherent sheaves on a smooth projective complex variety $Z$, denoted by $\Db(Z)$, is now widely recognized as an important invariant which can be used to study the geometry of $Z$. It is therefore quite natural to consider the group of autoequivalences $\Aut(\Db(Z))$. This group always contains the subgroup $\Aut^{\text{st}}(\Db(Z))$ of so-called standard autoequivalences, namely those generated by automorphisms of $Z$, the shift functor and tensor products with line bundles. Note that all these equivalences send coherent sheaves to (shifts of) coherent sheaves. A classical result by Bondal and Orlov, see \cite{Bondal-Orlov}, states that $\Aut^{\text{st}}(\Db(Z))\cong \Aut(\Db(Z))$ if $\omega_Z$ is either ample or anti-ample. Therefore, we expect the most interesting behaviour if the canonical bundle is trivial. For instance, if $\omega_Z\cong\ko_Z$ and $\Ho^i(Z,\omega_Z)=0$ for all $0<i<\dim(Z)$, then $\Aut(\Db(Z))$ contains so-called spherical twists, which are more interesting than the ones mentioned above since, in general, they do not preserve the abelian category of coherent sheaves; see \cite{Seidel-Thomas}.

It is usually quite difficult to construct new autoequivalences of a given variety. However, if $\ka$ is some triangulated category and an exact functor $\Phi\colon \ka \to \Db(Z)$ is a so-called spherical or $\IP^n$-functor, see Subsection \ref{sphericalpn}, we do get an autoequivalence of $\Db(Z)$. For example, let $\tilde{X}$ be a K3 surface and $\tilde{X}^{[n]}$ its Hilbert scheme of $n$ points. It was shown in \cite{Addington} that the Fourier--Mukai functor $\Db(\tilde{X})\to \Db(\tilde{X}^{[n]})$ with the ideal sheaf of the universal family as kernel is a $\IP^{n-1}$-functor whose associated autoequivalence is new. Interestingly, for an abelian surface $A$ the corresponding functor is not a $\IP^{n-1}$-functor, but pulling everything to the generalised Kummer variety does give one; see \cite{Meachan}. Another example was given in \cite{Krug1} where it was shown that, given any surface $S$, there is a $\IP^{n-1}$-functor $\Db(S)\to \Db(S^{[n]})$ which is defined using equivariant methods.
%as the composition of the functor $\Db(S)\to \Db_{S_n}(S)$ which equips any object with the trivial $S_n$-linearisation, the pushforward along the diagonal $\delta_*\colon \Db_{S_n}(S)\to \Db_{S_n}(S^n)$ and the Bridgeland--King--Reid--Haiman equivalence $\Db_{S_n}(S^n)\cong \Db(S^{[n]})$ is a $\IP^{n-1}$-functor.

In this paper we will construct new examples of spherical functors using Enriques surfaces and Hilbert schemes of points on them. Let $X$ be an Enriques surface and $X^{[n]}$ the Hilbert scheme of $n$ points on $X$. The canonical cover of $X^{[n]}$ is a Calabi--Yau variety (see \cite{Nieper} or \cite{Oguiso-Schroer}) and will be denoted by $\mathrm{CY}_n$. Write $\pi\colon \mathrm{CY}_n\to X^{[n]}$ for the quotient map. Consider the Fourier--Mukai functor $F\colon \Db(X)\to\Db(X^{[n]})$ induced by the ideal sheaf of the universal family.

\begin{thm}[Theorem \ref{maintheorem}]
The functor 
\[\tilde{F}=\pi^*F\colon \Db(X)\to \Db(\mathrm{CY}_n)\]
is split spherical for all $n\geq 2$ and the associated twist $\tilde{T}$ is equivariant, so descends to an autoequivalence of $X^{[n]}$. The autoequivalence $\tilde{T}$ of $\Db(\mathrm{CY}_n)$ is not standard and not a spherical twist.
\end{thm}

Under some conditions we can also compare our twist $\tilde{T}$ to the autoequivalences constructed in \cite{Ploog-Sosna}; see Proposition \ref{comparisontoboxes}.

Once we establish Theorem \ref{hilb-semiorth} below, the first part of the theorem is an incarnation of the following general principle, see Proposition \ref{sphericaldeg2cover}:

If $Y$ is a smooth projective variety whose canonical bundle is of order $2$, $\tilde{Y}$ its canonical cover with quotient map $\pi\colon \tilde{Y}\to Y$, and $\ka$ an admissible subcategory of $\Db(Y)$ with embedding functor $i\colon \ka\to \Db(Y)$, then $\pi^*i$ is a split spherical functor and the associated twist is equivariant. 

The following result might be of independent interest.

\begin{thm}\label{hilb-semiorth}
If $S$ is any surface with $p_g=q=0$, then the FM-transform $F\colon \Db(S)\to \Db(S^{[n]})$ whose kernel is the ideal sheaf of the universal family, is fully faithful, hence $\Db(S)$ is an admissible subcategory of $\Db(S^{[n]})$.
\end{thm}

Since there are many semi-orthogonal decompositions of $\Db(X^{[n]})$, we have many, potentially non-standard, twists associated with them.

The paper is organised as follows. In Section 2 we present some background information, before proving our main results in Section 3. In Section \ref{exceptionalsequences} we give a general construction of exceptional sequences on
the Hilbert scheme $S^{[n]}$ out of exceptional sequences on a surface $S$. This construction has been independently considered by Evgeny Shinder. In particular, we have the following result which is probably well-known to experts.

\begin{proposition}
If $S$ is a surface having a full exceptional collection, then the same holds for $S^{[n]}$.
\end{proposition}

In the last section we describe what we call truncated ideal functors which provide us with 
a further example of a fully faithful functor $\Db(X)\to \Db(X^{[n]})$ for $X$ an Enriques surface, and, in some cases, $\IP^n$-functors on smooth Deligne--Mumford stacks. The last section also gives some background on the proof of Proposition \ref{rfidentity}, the main ingredient in the proof of Theorem \ref{hilb-semiorth}.   
\smallskip

\noindent
\textbf{Conventions.} We will work over the complex numbers and all functors are assumed to be derived. We will write $\kh^i(E)$ for the $i$-th cohomology object of a complex $E\in \Db(Z)$ and $\Ho^*(E)$ for the complex $\oplus_i \Ho^i(Z,E)[-i]$. If $F$ is a functor, its right adjoint will be denoted by $F^R$ and the left adjoint by $F^L$.
\smallskip

\noindent
\textbf{Acknowledgements.} We thank Ciaran Meachan for comments. A.K.\ was supported by the SFB/TR 45 of the DFG (German Research Foundation). P.S.\ was partially financially supported by the RTG 1670 of the DFG.

\section{Preliminaries}\label{section:preliminaries}

\subsection{Hilbert schemes of surfaces with $p_g=q=0$}
Let $S$ be a surface with $p_g=q=0$ and consider $S^{[n]}$, the Hilbert scheme of $n$ points on $S$. Then we have $\Ho^k(S^{[n]},\ko_{S^{[n]}})=0$ for all $k>0$, compare \cite{Oguiso-Schroer}. 
Indeed, by Künneth formula $\Ho^*(S^n,\reg_{S^n})=\Ho^*(S,\reg_S)^{\otimes n}$ is concentrated in degree zero. 
As a consequence, the structure sheaf of the $n$-th symmetric product has no higher cohomology, and the same then also holds for $S^{[n]}$, because the symmetric product has rational singularities.

For example, we can consider an \emph{Enriques surface}, which is a smooth projective surface $X$ with $p_g=q=0$ such that the canonical bundle $\omega_X$ is of order $2$.

\subsection{Canonical covers} 
Let $Y$ be a variety with torsion canonical bundle of (minimal) order $k$. The \emph{canonical cover} $\tilde{Y}$ of $Y$ is the unique (up to isomorphism) variety with trivial canonical bundle and an \'etale morphism $\pi\colon \tilde{Y} \to Y$ of degree $k$ such that
 $\pi_*\ko_{\tilde{Y}}=\bigoplus_{i=0}^{k-1}\omega_Y^i$.
In this case, there is a free action of the cyclic group $\IZ/k\IZ$ on $\tilde{Y}$ such that $\pi$ is the quotient morphism.

As an example, the canonical cover of an Enriques surface $X$ is a K3 surface $\tilde{X}$ and $X$ is the quotient of $\tilde{X}$ by the action of a fixed point free involution. 

Furthermore, the canonical bundle of $X^{[n]}$ has order $2$ and the associated canonical cover, denoted here by $\mathrm{CY}_n$, is a Calabi--Yau variety; see \cite[Prop.\ 1.6]{Nieper} or \cite[Thm.\ 3.1]{Oguiso-Schroer}. %We will denote the quotient map by $\pi$.

\subsection{Fourier--Mukai transforms and kernels}\label{fmtransforms}
Recall that given an object $\ke$ in $\Db(Z\times Z')$, where $Z$ and $Z'$ are smooth and projective, we get an exact functor $\Db(Z)\to \Db(Z')$, $\alpha\mapsto (p_{Z'})_*(\ke\otimes p_Z^*\alpha)$. Such a functor, denoted by $\mathrm{FM}_\ke$, is called a \emph{Fourier--Mukai transform} (or FM-transform) and $\ke$ is its kernel. See \cite{Huybrechts} for a concise introduction to FM-transforms. For example, $\mathrm{FM}_{\Delta_*\kl}(\alpha)=\alpha\otimes\kl$, where $\Delta\colon Z\to Z\times Z$ for the diagonal map and $\kl\in \Pic(Z)$. In particular, $\mathrm{FM}_{\ko_\Delta}$ is the identity functor.
\smallskip

\noindent
{\textbf{Convention.}} We will write $\mathrm{M}_\kl$ for the functor $\mathrm{FM}_{\Delta_*\kl}$. 
\smallskip

Let $S$ be any smooth projective surface, $\kz_n\subset S\times S^{[n]}$ be the universal family and consider its structure sequence
\begin{align*}\xymatrix{0\ar[r] & \ki_{\kz_n}\ar[r] & \ko_{S\times S^{[n]}} \ar[r] & \ko_{\kz_n}\ar[r] & 0.}\end{align*}
Using the objects from the above sequence as kernels, we get a triangle $F\to F'\to F''$ of functors $\Db(S)\to \Db(S^{[n]})$. 
%and the functor considered in \cite{Addington} in the case where $S$ is a K3 surface is precisely $F$. 
Since all these functors are FM-transforms, they have left and right adjoints; see \cite[Prop.\ 5.9]{Huybrechts}.

\subsection{Equivalences of canonical covers}\label{ccequi}
The relation between autoequivalences of a variety $Y$ with torsion canonical bundle and those of the canonical cover $\tilde{Y}$ was studied in \cite{Bridgeland-Maciocia}. We recall some facts in the special case where the order of $\omega_Y$ is $2$.

An autoequivalence $\tilde{\varphi}$ of $\Db(\tilde{Y})$ is \emph{equivariant} under the conjugation action of $G=\IZ/2\IZ$ on $\Aut(\Db(\tilde{Y}))$ if there is a group automorphism $\mu\in \Aut(G)$ such that $g_*\tilde{\varphi}\cong\tilde{\varphi} \mu(g)_*$ for all $g\in G$. Of course, in our case, $\mu=\id$.

By \cite[Sect.\ 4]{Bridgeland-Maciocia}, an equivariant functor $\tilde{\varphi}$ descends to a functor $\varphi\in\Aut(\Db(Y))$ with functor isomorphisms
  $\pi_*\tilde{\varphi} \cong \varphi \pi_*$ and
  $\pi^* \varphi \cong \tilde{\varphi}\pi^*$;
moreover, two descents $\varphi$, $\varphi'$ of $\tilde{\varphi}$ are unique up to a line bundle twist with a power of $\omega_Y$. 
%As $\omega_Z$ is torsion, the line bundle twists $\M_{\omega_\bred Z}^i$ form a subgroup of $\DAut{Z}$ isomorphic to $G$. Hence, we get a group homomorphism $\bred{\DAutt{_\eq}\tZ} \to \DAut{Z}/G$. Note that the Serre functor $\M_{\omega_Z}[\dim Z]$ commutes with all autoequivalences, so that $G$ is central in $\DAut{Z}$, hence a normal subgroup.

In the other direction, it is also shown in \cite[Sect.\ 4]{Bridgeland-Maciocia} that every autoequivalence of $\Aut(\Db(Y))$ has an equivariant lift. Two lifts differ up to the action of $G$ in $\Aut(\Db(\tilde{Y}))$. 
%This induces a group homomorphism $\DAut{Z} \to \DAutt{_\eq}{\tZ}/G$.
%Furthermore, if two autoequivalences of $\D(Z)$ lift to the same equivariant autoequivalence, then they differ by a line bundle twist by a power of $\omega_Z$, i.e.\ the homomorphism above is a $k:1$-map. 

\subsection{Spherical functors}\label{sphericalpn}

Now consider two triangulated categories $\ka$ and $\kb$ and any exact functor $F\colon \ka \to\kb$ with left and right adjoints $F^L, F^R\colon\kb\to\ka$. Define the \emph{twist} $T$ to be the cone on the counit $\epsilon\colon FF^R\to \id_\kb$ of the adjunction and the \emph{cotwist} $C$ to be the cone on the unit $\eta\colon \id_\ka\to F^RF$. 

\begin{remark}\label{dglifts1}
Of course, one needs to make sure that the above cones actually exist. If one works with Fourier--Mukai-transforms, this is not a problem, because the maps between the functors come from the underlying kernels and everything works out, even for (reasonable) schemes which are not necessarily smooth and projective; see \cite{Anno-Logvinenko}. More generally, everything works out if one uses an appropriate notion of a spherical DG-functor; see \cite{Anno-Logvinenko2}. 
\end{remark}

So, as we will see, in the cases of interest to us, we have triangles $FF^R\to \id_\kb\to T$ and $\id_\ka\to F^RF\to C$. Following \cite{Anno-Logvinenko2}, we call $F$ \emph{spherical} if $C$ is an equivalence and $F^R\cong C F^L$. If $\ka$ and $\kb$ admit Serre functors $\ks_\ka$ and $\ks_\kb$, the last condition is equivalent to $\ks_\kb F C\cong F\ks_\ka$. If $F$ is a spherical functor, then $T$ is an equivalence. If the triangle $\id_\ka\to F^RF\to C$ splits, we call $F$ \emph{split spherical}. 

For an example of a (split) spherical functor consider a $d$-dimensional variety $Z$ and a \emph{spherical object} $E\in \Db(Z)$, that is, $E\otimes \omega_Z\cong E$ and $\Hom^*(E,E)\cong \IC\oplus\IC[-d]$. The functor 
\[F=-\otimes E\colon \Db(\Spec(\IC))\to \Db(Z)\]
is then spherical and the associated autoequivalence of $\Db(Z)$ is the spherical twist from the introduction, denoted by $\mathrm{ST}_E$ in the following.

\subsection{$\IP^n$-functors} 
Following \cite{Addington}, a \emph{$\IP^n$-functor} is a functor $F\colon \ka\to \kb$ of triangulated categories such that
\begin{enumerate}
 \item There is an autoequivalence $H_F=H$ of $\ka$ such that 
\[F^RF\simeq \id\oplus H\oplus H^2\oplus\dots \oplus H^n.\]
\item The map 
\[HF^RF\hookrightarrow F^RFF^RF\xrightarrow{F^R\epsilon F} F^RF,\]
with $\epsilon$ being the counit of the adjunction is, when written in the components
%\label{compo}
\begin{align*}H\oplus H^2\oplus\dots\oplus H^n\oplus H^{n+1}\to \id\oplus H\oplus H^2\oplus\dots\oplus H^n,\end{align*}
of the form
\begin{align}\label{monadmatrix}\begin{pmatrix}
  * & * &\cdots &*&*\\
1&*&\cdots&*&*\\
0&1&\cdots&*&*\\
\vdots&\vdots&\ddots&\vdots&\vdots\\
0&0&\cdots&1&* 
  \end{pmatrix}.
  \end{align}
\item $F^R\simeq H^nF^L$. If $\ka$ and $\kb$ have Serre functors, this is equivalent to $S_\kb FH^n\simeq FS_\ka$.
\end{enumerate}

If $F$ is a $\IP^n$-functor, then there is also an associated autoequivalence of $\kb$, denoted by $P_F=P$. A $\IP^1$-functor is precisely a split spherical functor and for the associated equivalences we have $T^2\cong P$. 
If $\tilde{X}$ is a K3 surface, the functor $F=\mathrm{FM}_{\ki_{\kz_n}}$ defined in Subsection \ref{fmtransforms} is a $\IP^{n-1}$-functor; see \cite{Addington}.

\subsection{Values of autoequivalences}
Let $Z,Z'$ be smooth projective varieties. Recall that an exact functor $F\colon \Db(Z)\to \Db(Z')$ is said to have \emph{cohomological amplitude} $[a,b]$ if for every complex $E\in \Db(Z)$ whose cohomology is concentrated in degrees between $p$ and $q$, the cohomology of $F(E)$ is concentrated in degrees between $p-a$ and $q+b$.

We will need the following slight generalisation of a Proposition in {\cite[Sect.\ 1.4]{Addington}}. The case that $\Phi_1=\Phi_2=\id$ could be generally useful when comparing twists along spherical and $\IP^n$-functors $F$ whose cotwist is not simply a shift but a standard autoequivalence. 

\begin{proposition}\label{prop:Addingtongeneralised}
Let $T\in\Aut(\Db(Z'))$ and $F\colon\Db(Z)\to \Db(Z')$ be a Fourier--Mukai transform. Furthermore, let $\Psi\in \Pic(Z)\rtimes\Aut(Z)\subset\Aut^{\mathrm{st}}(\Db(Z))$ be a non-shifted standard autoequivalence and let $\Phi_1,\Phi_2\in \{\id,\mathrm{M}_{\omega_{Z'}}\}\subset\Aut(\Db(Z'))$. 
If $\alpha\in \Db(Z')$ is such that $T(\alpha)\cong\Phi_1(\alpha)[i]$ and there is an isomorphism of functors $T F\cong\Phi_2 F \Psi[j]$ with $i\neq j\in \mathbb{Z}$, then $F(\beta)\in \alpha^\bot$ and $\alpha \in F(\beta)^\bot$ for every $\beta\in \Db(Z)$.
\end{proposition}
\begin{proof}
We have $T^m(\alpha)\cong\Phi_1^m(\alpha)[mi]$ and $T^mF(\beta)\cong\Phi_2^mF\Psi^m(\beta)[mj]$ for all $m\in \IZ$, since $\mathrm{M}_{\omega_Z}$ commutes with every autoequivalence. Hence,
\begin{align*}
\Hom\bigl(\alpha,F(\beta)[k]\bigr)&\cong\Hom\bigl(T^m(\alpha),T^mF(\beta)[k]\bigr)\\&\cong\Hom\bigl(\Phi_1^m(\alpha)[mi],\Phi_2^mF\Psi^m(\beta)[mj+k]\bigr)
\\&\cong\Hom\bigl(\Psi^{-m}F^L\Phi_2^{-m}\Phi_1^m(\alpha),\beta[(j-i)m+k]\bigr). 
\end{align*}
for every $k\in \IZ$. This vanishes for $m\gg0$ since $\Phi_1$, $\Phi_2$, and $\Psi$ have cohomological amplitude $[0,0]$ and $F^L$ has finite cohomological amplitude by \cite[Prop.\ 2.5]{Kuz}. The proof that $\alpha\in F(\beta)^\bot$ is analogous.
\end{proof}

\subsection{Semi-orthogonal decompositions} References for the following facts are, for example, \cite{Bon} and \cite{Bondal-Orlov2}. 

Let $\kt$ be a triangulated category. A \emph{semi-orthogonal decomposition} of $\kt$ is a sequence of strictly full triangulated subcategories $\ka_1,\ldots,\ka_m$ such that (a) if $A_i\in \mathcal{A}_i$ and $A_j\in \mathcal{A}_j$, then $\text{Hom}(A_i,A_j[l])=0$ for $i>j$ and all $l$, and (b) the $\mathcal{A}_i$ generate $\kt$, that is, the smallest triangulated subcategory of $\kt$ containing all the $\mathcal{A}_i$ is already $\kt$. We write $\kt=\langle\ka_1,\ldots,\ka_m\rangle$. If $m=2$, these conditions boil down to the existence of a functorial exact triangle $A_2\to T\to A_1$ for any object $T\in \kt$.

A subcategory $\ka$ of $\kt$ is \emph{right admissible} if the embedding functor $i$ has a right adjoint $i^R$, and it is called \emph{left admissible} if $i$ has a left adjoint $i^L$. We say that $\ka$ is \emph{admissible} if both adjoints exist. Note that if $\kt$ admits a Serre functor, then the existence of one of the adjoints implies the existence of the other. 

Given any subcategory $\ka$, the category $\ka^\bot$ consists of objects $b$ such that $\Hom(a,b[k])=0$ for all $a\in \ka$ and all $k\in \IZ$. If $\ka$ is right admissible, then $\kt=\langle\ka^\bot,\ka\rangle$ is a semi-orthogonal decomposition. Similarly, $\kt=\langle\ka, {}^\bot\ka\rangle$ is a semi-orthogonal decomposition if $\ka$ is left admissible, where ${}^\bot\ka$ is defined in the obvious way.

Examples typically arise from so-called exceptional objects. Recall that an object $E\in \Db(Z)$ (or any $\IC$-linear triangulated category) is called \emph{exceptional} if $\Hom(E,E)=\IC$ and $\Hom(E,E[k])=0$ for all $k\neq 0$. The smallest triangulated subcategory containing $E$ is then equivalent to $\Db(\Spec(\IC))$ and this category, by abuse of notation again denoted by $E$, is admissible, leading to a semi-orthogonal decomposition $\Db(Z)=\langle E^\bot, E\rangle$. We call a sequence of exceptional objects $E_1,\ldots,E_n$ an \emph{exceptional collection} if $\Db(Z)=\langle (E_1,\ldots,E_n)^\bot,E_1,\ldots,E_n\rangle$, where $(E_1,\ldots,E_n)^\bot$ is the category of objects $F$ which satisfy $\Hom(E_i,F[k])=0$ for all $i,k$. The collection is called \emph{full} if $(E_1\ldots,E_n)^\bot=0$.

Note that any fully faithful FM-transform $i\colon\ka=\Db(Z')\to \Db(Z)$ gives a semi-orthogonal decomposition $\Db(Z)=\langle i(\ka)^\bot,i(\ka)\rangle$. 

We will need the following well-known and easy fact.

\begin{lemma}\label{serreadmissible}
If $\kt$ has a Serre functor $\ks_\kt$ and $\ka$ is an admissible subcategory, then $\ka$ has a Serre functor $\ks_\ka\cong i^R\ks_\kt i$.
\end{lemma}  

\begin{proof}
Given $a,a'\in \ka$, we compute $\Hom_\ka(a,a')\cong\Hom_\kt(i(a),i(a'))\cong \Hom_\kt(i(a'),\ks_\kt i(a))^\vee\cong \Hom_\ka(a',i^R\ks_\kt   i(a))^\vee$.
\end{proof}

\begin{remark}\label{dglifts2}
Assume $\ka$ is an admissible subcategory of $\Db(Z)$. The embedding functor $i$ lifts to DG-enhancements (see, for example, \cite{Lunts-Kuznetsov} for this notion). It can be checked that the adjoints $i^R$ and $i^L$ also lift to so-called DG quasi-functors, which follows, for example, from \cite[Lem.\ 4.4 and Prop.\ 4.10]{Lunts-Kuznetsov}. So the composition of $i$ with any FM-transform will lift to the DG-level, hence we are in a position to use the results of \cite{Anno-Logvinenko2} and all required cones of natural transformations will exist.
\end{remark}

\subsection{Group actions and derived categories}\label{subsection:equivariant} 

Let $G$ be a finite group acting on a smooth projective variety $Z$. The \emph{equivariant derived category}, denoted by $\Db_G(Z)$, is defined as $\Db(\Coh^G(Z))$, see, for example, \cite{Ploog} for details.
Recall that for every subgroup $H\subset G$ the restriction functor $\Res\colon \Db_G(Z)\to \Db_H(Z)$ has the inflation functor $\Inf\colon \Db_H(Z)\to \Db_G(Z)$ as a left and right adjoint (see e.g.\ \cite[Sect.\ 1.4]{Ploog}). It is given for $A\in \Db(Z)$ by
\begin{align}\label{infdef} \Inf(A)=\bigoplus_{[g]\in H\setminus G}g^* A
\end{align}
with the linearisation given by permutation of the summands.

If $G$ acts trivially on $Z$, there is also the functor $\triv\colon\Db(Z)\to \Db_{G}(Z)$ which equips an object with the trivial $G$-linearisation. Its left and right adjoint is the functor $(-)^G\colon \Db_G(Z)\to \Db(Z)$ of invariants. 

\smallskip

\noindent
\textbf{Convention.} When working with Fourier--Mukai transforms, we will frequently identify the functor with its kernel.

\section{Proofs of the main results}\label{section:proofs}

\subsection{Surfaces with $p_g=q=0$}

Recall the FM-transforms from Subsection \ref{fmtransforms}. To compute $RF$ in the examples known so far, one usually works out the various compositions such as, for example, $R''F'$. This can be done rather quickly in certain situations:

\begin{proposition}
Let $S$ be a surface with $p_g=q=0$. Then the following holds:
\begin{align*}
R''F'&\cong \ko_{S\times S},\\
R'F'&\cong (\ko_S\boxtimes \omega_S)[2],\\
R''F''&\cong \ko_\Delta\oplus \ko_{S\times S},\\
R'F''&\cong (\ko_S\boxtimes \omega_S)[2].
\end{align*}
\end{proposition}

\begin{proof}
Follows immediately from the results in Section 6 of \cite{Meachan} by using $\Ho^*(S^{[n]}, \ko_{S^{[n]}})\cong \IC$.
\end{proof}

%\begin{corollary}\label{rfdoubleprime}
%We have $RF''\cong (\ko_\Delta\oplus \ko_{S\times S})[1]\oplus \ko_S\boxtimes \omega_S[2]$ and $RF'\cong \ko_{S\times S}[1]\oplus \ko_S\boxtimes \omega_S[2]$.
%\end{corollary}
%
%\begin{proof}
%We consider the cohomology sequence of the triangle 
%\[R''F''\to R'F''\to RF''\] 
%which immediately gives that $\kh^{-2}(RF'')\cong \ko_S\boxtimes \omega_S$ and $\kh^{-1}(RF'')\cong \ko_\Delta\oplus \ko_{S\times S}$. Since $\Ext^1( \ko_\Delta\oplus \ko_{S\times S}[-1],\ko_S\boxtimes \omega_S[-2])\cong 0$, the filtration of $RF''$ by its cohomology objects splits and the first statement follows.
%
%For the second, we look at the triangle $R''F'\to R'F'\to RF'$ which gives $\kh^{-2}(R'F)\cong\ko_S\boxtimes \omega_S$ and $\kh^{-1}(R'F)\cong \ko_{S\times S}$. The same argument as above concludes the proof. 
%\end{proof}

\begin{lemma}\label{rprimef}
We have $R'F\cong 0$ and $R''F\simeq \reg_\Delta[-1]$.
\end{lemma}

\begin{proof}
The map $F'\to F''$ induces an isomorphism $R'F'\to R'F''$ and the component $\reg_{S\times S}\to \reg_{S\times S}$ of the induced map $R''F'\to R'F''$ is an isomorphism too; see \cite[Sect.\ 6]{Meachan} or Section \ref{truncated}. Hence, the first assertion follows from the triangle $R'F\to R'F'\to R'F''$. 

We then consider the triangle $R''F\to R''F'\to R''F''$ and check that the cokernel of the map $R''F'\to R''F''$ is isomorphic to $\ko_\Delta$. Indeed, if $\varphi=(\varphi_1,\varphi_2)\colon A\to A\oplus B$ is a map in an abelian category such that the first component is an isomorphism (so $\varphi$ is an injection), the cokernel has to be isomorphic to $B$. To see this, embed the situation into the abelian category of modules over some ring using the Freyd--Mitchell theorem and define $c\colon A\oplus B\to B$ by $(a,b)\mapsto b-\varphi_2(\varphi_1^{-1}(a))$. It is clear that $c\circ\varphi=0$. Now given a morphism $f\colon A\oplus B\to C$ such that $f\circ\varphi=0$, define $g\colon B\to C$ by $b\mapsto f(0,b)$. It is then straightforward to check that $g\circ c=f$.
\end{proof}

%We compute
%\begin{align*}
%\Hom(\pi_2^*\omega_S,\pi_2^*\omega_S)&\cong \Hom(\omega_S,(\pi_2)_*\pi_2^*\omega_S)\cong \Hom(\omega_S,\omega_S\otimes (\pi_2)_*\ko_{S\times S})\\
%&\cong \Hom(\omega_S,\omega_S)\cong \Hom(\ko_S,\ko_S)\cong H^0(S,\ko_S)\cong \IC,
%\end{align*}
%where we used that $(\pi_2)_*\ko_{S\times S}\cong \ko_S$ since $(\pi_2)_*\pi_1^*\ko_S\cong H^*(\ko_S)\otimes \ko_S$. 

\begin{proposition}\label{rfidentity}
The composition $RF$ is isomorphic to the identity.
\end{proposition}

\begin{proof}
Use the triangle $R''F\to R'F\to RF$ and the above lemma. 
\end{proof}

We are now ready to show our first main result.

\begin{proof}[Proof of Theorem \ref{hilb-semiorth}]
Since $RF\cong \id$, $F$ is fully faithful. On the other hand, $F$ has adjoints, so $F(\Db(S))$ is an admissible subcategory of $\Db(S^{[n]})$.
\end{proof}

\begin{remark}
The above shows that for any surface with $p_g=q=0$, the functor $F$ is quite far from being a spherical or a $\IP^n$-functor.
\end{remark}

\begin{remark}
There exist surfaces $S$ of general type with $p_g=q=0$ such that $\Db(S)$ contains an admissible subcategory whose Hochschild homology is trivial and whose Grothendieck group is finite or torsion, see, for example, \cite{BBKS12}. Therefore, by Theorem \ref{hilb-semiorth}, $\Db(S^{[n]})$ also contains such a (quasi-)phantom category.
\end{remark}

\subsection{Canonical bundles of order two and spherical functors}

Consider a $d$-dimensional variety $Y$ whose canonical bundle is of order $2$ and its canonical cover $\pi\colon\tilde{Y}\to Y$ with deck transformation $\tau\colon \tilde Y\to \tilde Y$.

\begin{lemma}
The functor $\pi^*\colon\Db(Y)\to \Db(\tilde Y)$ is split spherical with cotwist $C=(-)\otimes \omega_Y$ and twist $T_{\pi^*}=\tau^*[1]$. 
\end{lemma}

\begin{proof}
This follows from the identities $\pi_*\pi^*\cong (-)\otimes (\reg_Y\oplus \omega_Y)$, $\pi^*\omega_Y\cong \omega_{\tilde Y}$, and $\pi^*\pi_*\cong \id\oplus \tau^*$. 
\end{proof}

\begin{proposition}\label{sphericaldeg2cover}
 If $\ka$ is an admissible subcategory of $\Db(Y)$ with embedding functor $i$, then the functor $\pi^*i\colon \ka\to\Db(\tilde{Y})$ is split spherical. The associated twist $\tilde{T}_{\ka}:=T_{\pi^*i}$ is equivariant. 
\end{proposition}

\begin{proof}
That $\pi^*i$ is split spherical follows by the previous lemma together with Lemma \ref{serreadmissible}; compare \cite[Prop.\ on p.7]{Addington}.

To see that $\tilde{T}_\ka$ is equivariant, we note that $\pi^*ii^R\pi_*\tau_*\cong \pi^*ii^R\pi_*\cong \tau_*\pi^*ii^R\pi_*$, so $\tau_*\tilde{T}_\ka\cong \tilde{T}_\ka \tau_*$, since both are a cone of $\pi^*ii^R\pi_*\to \tau_*$.
\end{proof}

\begin{example}
If $\ka\cong \Db(\Spec(\IC))$ is the category generated by an exceptional object $E$, then the twist associated to $\pi^*i$ is the spherical twist $\mathrm{ST}_A$, where $A\cong \pi^*i(E)$ is a spherical object by \cite[Prop.\ 3.13]{Seidel-Thomas}.
\end{example}

\begin{remark}
Let $\Db(Y)=\langle \ka,\kb\rangle$ be a semi-orthogonal decomposition.
By \cite[Thm.\ 11]{Addington-Aspinwall} we have $\tilde T_{\ka}\tilde T_{\kb}\cong \tau^*[1]$.
\end{remark}

\begin{proposition}\label{descentdesc}
If $\langle \ka,\ka\otimes \omega_Y\rangle ^\perp\neq 0$ and $\ka\cong \Db(Z)$ for some smooth projective variety with $\dim Z\le d-2$, then $\tilde{T}_\ka$ and its descents are non-standard equivalences of $\Db(\tilde{Y})$ and $\Db(Y)$, respectively.
\end{proposition}

\begin{proof}
One of the two descents of $\tilde{T}_\ka$ is 
\begin{align}\label{coneofT} T_\ka:=\cone( i i^R\oplus \mathrm{M}_{\omega_Y}i i^R \mathrm{M}_{\omega_Y} \xrightarrow{c} \id)\end{align}
where the components of $c$ are given by the counits of the adjunctions. We also get  
\begin{align}\label{Taction}T_\ka  i\cong \mathrm{M}_{\omega_Y} i \ks_\ka[-d+1]\,,\quad T_\ka \mathrm{M}_{\omega_Y} i\cong i S_\ka[-d+1]\end{align} 
by Lemma \ref{serreadmissible}. In particular, for every $\mathrm{M}_{\omega_Z}$-invariant object $\alpha\in \D^b(Z)$, for example a skyscraper sheaf, $T_\ka i(\alpha)=\alpha[\dim Z-d+1]$.
Furthermore, $T_\ka$ acts as the identity on $\langle \ka,\ka\otimes \omega_Y\rangle ^\perp$.
Thus, $T_\ka$ has cohomological amplitude at least $[\dim Z-d+1,0]$. In contrast, the cohomological amplitude of a standard autoequivalence is of the form $[c,c]$ for some $c\in \IZ$. 
\end{proof}

\begin{remark}
By the previous proof we have a description of the restriction of $T_\ka$ to $\kc:=\ka\cup \ka\otimes \omega_Y \cup \langle\ka, \ka\otimes\omega_Y\rangle^\perp$. Note that is a spanning class (see \cite[Sect.\ 1.3]{Huybrechts} for details regarding this notion). Indeed, if $\Hom(\beta,\gamma)=0$ for all $\gamma\in \kc$, then, in particular, $\Hom(\beta,\alpha)=0=\Hom(\beta,\alpha\otimes\omega_Y)$ for all $\alpha\in \ka$. By Serre duality, $\beta\in\langle\ka, \ka\otimes\omega_Y\rangle^\perp\subset \kc$, hence $\beta\cong 0$. Similarly, one proves that $\Hom(\gamma,\beta)=0$ implies that $\beta\cong 0$. 
\end{remark}

%\begin{remark}\label{descentdesc}
%Since $\tilde T_{\ka}$ is equivariant, it descends to autoequivalences of $\Db(Y)$ as explained in section \ref{ccequi}. One of the two descents is 
%\[T_\ka:=\cone( i i^R\oplus U_{\omega_Y}i i^R U_{\omega_Y} \xrightarrow{c} \id)\]
%where $U_{\omega_Y}$ denotes the involution $(-)\otimes \omega_Y$ and the components of $c$ are given by the counits of the adjunctions. We see that the restriction of $T_\ka$ to $\kb$ equals the left-mutation with respect to $\ka$.
%We also get  
%\begin{align}\label{Taction}T_\ka \circ i\cong U_{\omega_Y}\circ i\circ \ks_\ka[-d+1]\,,\quad T_\ka\circ U_{\omega_Y}\circ i\cong i\circ S_\ka[-d+1]\end{align} 
%by Lemma \ref{serreadmissible}.
%Furthermore, $T_\ka$ acts as the identity on $\langle \ka,\ka\otimes \omega_Y\rangle ^\perp$.
%So we have a description of $T_\ka$ on the spanning class $\ka\cup \ka\otimes \omega_Y \cup \langle\ka, \ka\otimes\omega_Y\rangle^\perp$.  
%We also see that $T_\ka$ as well as $\tilde T_\ka$ are non-standard as soon as $\langle \ka,\ka\otimes \omega_Y\rangle ^\perp\neq 0$ and $\ka\cong \Db(Z)$ for some smooth projective variety with $\dim Z\le d-2$.   
%\end{remark}

\begin{remark}\label{twistrel}
As in the case of spherical or $\IP^n$-twists (see \cite[Lem.\ 2.3]{Krug1}), we have for any $\Psi\in \Aut(\Db(Y))$ the relation $\Psi T_\ka\cong T_{\Psi(\ka)}\Psi$. For this one uses the cone description of $T_\ka$ given by Equation (\ref{coneofT}) and the fact that the embedding functor of $\Psi(\ka)$ is $\Psi i$.
\end{remark}

\begin{remark}
More generally, $\pi^*$ is a $\IP^{n-1}$-functor if $Y$ has torsion canonical bundle of order $n\geq 2$. But the analogue of Proposition \ref{sphericaldeg2cover} does not hold for $n\ge 3$, that is, in general, for a fully faithful admissible embedding $i\colon \ka\to \Db(Y)$ the composition $\pi^*i$ is not a $\IP^{n-1}$-functor. The reason is that Lemma \ref{serreadmissible} does not generalise to powers of the Serre functor, that is, in general, $\ks^k_\ka\not\simeq i^R\ks^k_\kt i$ for $k\geq 2$.  
\end{remark}

\subsection{Application to Enriques surfaces}

In the following we want to apply Proposition \ref{sphericaldeg2cover} to the Hibert scheme of points on an Enriques surface and investigate the relation of the associated twist with some known autoequivalences. For this we first need the following statement. 

\begin{lemma}\label{rank}
Let $S$ be any surface and let $\xi\in S^{[n]}$ be a point representing $n$ pairwise distinct points on $S$. Then $FR(k(\xi))$ has  rank $\chi-2n$ where $\chi:=\chi(\omega_S)=\chi(\reg_S)$.
\end{lemma}

\begin{proof}
We will follow the computations in \cite[Lem.\ 5.7]{Krug1}.
So let $\xi \in S^{[n]}$ correspond to $n$ pairwise distinct points and write $\alpha=k(\xi)$. Recall that $F'$ has kernel $\ko_{S\times S^{[n]}}$ and $F''$ has kernel $\ko_{\kz_n}$. The kernels of the right adjoints $R'$ and $R''$ are then given by $p_S^*\omega_S[2]$ and $\ko_{\kz_n}^\vee\otimes p_S^*\omega_S[2]$, respectively. Hence, $R'(\alpha)=\omega_S[2]$ and $R''(\alpha)\cong \ko_\xi\otimes\omega_S\cong\ko_\xi$, since outside the locus of points representing non-reduced subschemes of $S$ $\ko_{\kz_n}^\vee$ is a line bundle shifted into degree 2. We have $F'=\Ho^*(-)\otimes \reg_S$. It follows that $\rk F'(\omega_S)=\chi$ and $\rk F'(\reg_\xi)=n$. Since $\kz_n$ is flat and finite of degree $n$ over $S^{[n]}$, we have 
$\rk F''(\omega_S)=n$ and $\rk F''(\reg_\xi)=0$. Using that the rank is compatible with exact triangles, we get
\begin{align*}
 \rk FR(\alpha)&=\rk F'R'(\alpha)-\rk F'R''(\alpha)-\rk F''R'(\alpha)+\rk F''R''(\alpha)\\&=\chi -n-n+0=\chi-2n.
\end{align*}
\end{proof}

\begin{corollary}\label{frsupport}
Let $S$ be a surface with $p_g=q=0$. Then the support of $FR(k(\xi))$ is $S^{[n]}$.\qqed
\end{corollary}

\begin{thm}\label{maintheorem}
Let $X$ be an Enriques surface, $F\colon \Db(X)\to \Db(X^{[n]})$ the FM-transform induced by the ideal sheaf of the universal family, $\mathrm{CY}_n$ the canonical cover of $X^{[n]}$ and $\ka\subset \Db(X)$ an admissible subcategory with embedding functor $i$. Then $\pi^*Fi$ is a split spherical functor whose induced twist is equivariant for all $n\geq 2$.

If $\ka=\Db(X)$, then the twist $\widetilde{T}=\tilde{T}_{\Db(X)}$ associated to $\widetilde{F}=\pi^*F$ is a non standard autoequivalence of $\mathrm{CY}_n$. Furthermore, it is not a spherical twist.
\end{thm}

\begin{proof}
The first part is just an application of Proposition \ref{sphericaldeg2cover} together with Theorem \ref{hilb-semiorth}, so let us investigate $\tilde{T}$.

By Corollary \ref{frsupport}, $FR(k(\xi))$ is supported on $X^{[n]}$, hence the same holds for the image under $\tilde{F}\tilde{R}$ of a skyscraper sheaf of a point on $\mathrm{CY}_n$ mapping to $\xi$. Using the triangle defining $\tilde{T}$ and the fact that for any triangle $\alpha\to \beta\to \gamma$ in $\Db(Z)$, we have $\supp(\alpha)\subset \supp(\beta)\cup\supp(\gamma)$, we conclude that $\tilde{T}$ does not respect the dimension of the support. Therefore, $\tilde{T}$ cannot be a standard autoequivalence. 

Now, if $A$ is a spherical object in $\mathrm{CY}_n$, then the spherical twist $\mathrm{ST}_A$ acts as identity on $A^\bot$ and $\mathrm{ST}_A(A)\cong A[1-2n]$. The discussion on p.\ 9 in \cite{Addington} shows that $\widetilde{T}\widetilde{F}(\alpha)\cong\widetilde{F}C[1]\cong\widetilde{F}\mathrm{M}_{\omega_X}[3-2n]$. If $\widetilde{T}$ were isomorphic to $\mathrm{ST}_A$ for some $A$, then, by Proposition \ref{prop:Addingtongeneralised}, $\widetilde{F}(\alpha)$ would have to be orthogonal to the spanning class $A\cup A^\bot$ (see \cite[Prop.\ 8.6]{Huybrechts} for details about this spanning class), hence zero for all $\alpha\in \Db(X)$. But clearly $\widetilde{F}$ is not trivial, a contradiction.
\end{proof}

\begin{remark}\label{kerR}
The above shows that there are objects which $\tilde{T}$ shifts by $[3-2n]$. On the other hand, we know that $\tilde{T}$ acts as identity on $\ker\widetilde{R}$. However, finding a non-trivial element in $\ker\widetilde{R}$ seems to be difficult.
\end{remark}

\begin{remark}
The functor $\widehat{F}\colon \Db(\tilde{X})\to \Db(X^{[n]})$ defined as $F \pi_{X*}$ is not spherical. Note that $\pi_X^!=\pi_X^*$ in this case, so $\widehat{R} \widehat{F}\cong \pi_X^* \pi_{X*}\cong \id\oplus \tau_X^*$ and $C\cong \tau_X^*$ is an autoequivalence of $\Db(\tilde{X})$. But the condition $\widehat{F} \ks_{\tilde{X}}\cong \ks_{X^{[n]}} \widehat{F} C$ is not satisfied. Indeed, $\widehat{F} \ks_{\tilde{X}}(\alpha)\cong F(\pi_*\alpha)$ and $\ks_{X^{[n]}} \widehat{F} C$ are non isomorphic objects for $\alpha\in \D^b(\tilde X)$ since their non-vanishing cohomologies lie in different degrees.  

Furthermore, the functor $\overline{F}\colon \Db(\tilde{X})\to \Db(\mathrm{CY}_n)$ defined as $\pi^* F \pi_{X*}$ is also not spherical. Indeed, $\overline{R}=\pi_X^* R\pi_*$, hence 
\begin{align*}
\overline{R}\,\overline{F}&\cong \pi_X^*\tilde{R}\tilde{F}\pi_{X*}\cong \pi_X^*\pi_{X*}\oplus \pi_X^*\pi_{X*}[2-2n]\\
&\cong \id\oplus \tau_X^*\oplus[2-2n]\oplus \tau_X^*[2-2n].
\end{align*}
\end{remark}

We can say something about the autoequivalences on $X^{[n]}$ the twist $\tilde{T}$ descends to.

\begin{proposition}
Let $T'\in \Aut(\Db(X^{[n]}))$ be a descent of $\tilde{T}$. Then $T'$ is not contained in the subgroup of $\Aut(\Db(X^{[n]}))$ generated by $\Aut^{\rm{st}}(\Db(X^{[n]}))$ and the equivalence $P$ arising from the $\IP^{n-1}$-functor constructed in \cite{Krug1}. 
\end{proposition}

\begin{proof}
The twist $P$ is rank-preserving since all the objects in the image of the corresponding $\IP^{n-1}$-functor are supported on a proper subset of $X^{[n]}$; see \cite[Rem.\ 4.7]{Krug1}. The same holds for every standard autoequivalence (up to the sign -1 occurring for odd shifts). But by 
Lemma \ref{rank} and Equation (\ref{coneofT}) we have $\rk(T(k(\xi))=4n-2$, where $T=T_{\Db(X)}$. Since the other descent is given by $\mathrm{M}_{\omega_{X^{[n]}}}T$, the claim follows.   
\end{proof}

\begin{lemma}\label{notwist}
There is no $0\neq \alpha\in \Db(X^{[n]})$ such that a descent $T'$ of $\tilde{T}$ satisfies $T'(\alpha)=\alpha[\ell]$ or $T'(\alpha)=\alpha\otimes \omega_{X^{[n]}}[\ell]$ for $\ell\notin\{0,3-2n\}$. There is also no $0\neq\tilde \alpha\in \Db(\mathrm{CY}_n)$ such that $\tilde T(\tilde \alpha)=\tilde \alpha[\ell]$ for $\ell\notin\{0,3-2n\}$.
\end{lemma}
\begin{proof}
It is sufficient to consider the descent $T$ described in Proposition \ref{descentdesc}. Indeed, the other descent is given by $\mathrm{M}_{\omega_{X^{[n]}}}T$. 
%So for any $\alpha$ such that $\mathrm{M}_{\omega_{X^{[n]}}}T(\alpha)=\alpha[\ell]$ we have that $T(\alpha\oplus \alpha\otimes \omega_{X^{[n]}})=(\alpha\oplus\alpha\otimes \mathrm{M}_{\omega_{X^{[n]}}})[\ell]$, since every autoequivalence commutes with $\mathrm{M}_{\omega_{X^{[n]}}}$.
 
Let $\alpha\in \D^b(X^{[n]})$ with $T(\alpha)=\alpha[\ell]$, $\ell\notin\{0,3-2n\}$. 
By Equation (\ref{Taction}) and Proposition \ref{prop:Addingtongeneralised} (here we use that $l\neq 3-2n$) we get that $\alpha\in F(\beta)^\perp$ and $F(\beta)\in \alpha^\perp$ for all $\beta\in \Db(X)$. Hence $\alpha\in \ka^\perp$ and $\alpha\in {}^\bot\ka$, where $\ka=F(\Db(X))$. By Serre duality also $\alpha\in (\ka\otimes \omega_{X^{[n]}})^\perp$ and ${}^\bot(\ka\otimes \omega_{X^{[n]}})$. 

The object $\alpha$ is also orthogonal to $\langle \ka,\ka\otimes \omega_{X^{[n]}}\rangle^\perp$ on which $T$ acts trivially. 
To see this, use again Proposition \ref{prop:Addingtongeneralised}, this time with $F$ being the embedding of the subcategory $\langle \ka,\ka\otimes \omega_{X^{[n]}}\rangle^\perp$. 

Therefore, $\alpha$ is orthogonal to the spanning class $\ka\cup \ka\otimes\omega_{X^{[n]}}\cup \langle \ka,\ka\otimes \omega_{X^{[n]}}\rangle^\perp$, hence zero. The proof in the case that $T(\alpha)=\alpha\otimes \omega_{X^{[n]}}[\ell]$ is similar and the statement about $\tilde{\alpha}$ follows by applying $\pi_*$. 
\end{proof}

Next, we want to compare our equivalences to the ones constructed in \cite{Ploog-Sosna}. For this we need to recall some facts which will also be useful in the next section.

Let $Z$ be a smooth projective variety and $n\ge 2$. We consider the cartesian power $Z^n$ equipped with the natural $\sym_n$-action given by permuting the factors.

For $E\in \Db(Z)$ an exceptional object, the box product $E^{\boxtimes n}$ is again exceptional, since by Künneth formula
\[
 \Ext^*_{\Db_{\sym_n}(Z^n)}(E^{\boxtimes n},E^{\boxtimes n})\cong \Ext^*_{X^n}(E^{\boxtimes n},E^{\boxtimes n})^{\sym_n}\cong S^n\Ext^*_X(E,E)\cong \IC[0]. 
\]
More generally, for $\rho$ an irreducible representation of $\sym_n$, the object $E^{\boxtimes n}\otimes \rho$ is exceptional and the objects obtained this way are pairwise orthogonal, i.e. $\Ext^*(E^{\boxtimes n}\otimes \rho,E^{\boxtimes n}\otimes \rho')=0$ for $\rho\neq\rho'$. 

The case of biggest interest is if $Z$ is a surface. Then there is the Bridgeland--King--Reid--Haiman equivalence (see \cite{BKR} and \cite{Hai}) 
\[\Phi\colon \Db(Z^{[n]})\xrightarrow{\simeq} \Db_{\sym_n}(Z^n).\]
In particluar, when $Z=X$ is an Enriques surface, we get further induced autoequivalences of $X^{[n]}$ and its cover $\mathrm{CY}_n$ using Proposition \ref{sphericaldeg2cover}.

Let now $Z=X$ be an Enriques surface and assume that there exists an object $0\neq F\in \langle E, E\otimes \omega_X\rangle ^\perp$. This is equivalent to $\tilde E^\perp$ being non-trivial, where $\tilde E:=\pi^*E$ is the corresponding spherical object on the K3 cover of $X$. 
In this case $F^{\boxtimes n}$ is orthogonal to every $E^{\boxtimes n}\otimes \rho$ and $\omega_{X^n}\otimes E^{\boxtimes n}\otimes \rho$. Thus, by Proposition \ref{descentdesc}, all the associated twists $T_\rho:=T_{E^{\boxtimes n}\otimes \rho}\in\Aut( \Db(X^{[n]}))$ and $\tilde T_\rho:=\tilde T_{E^{\boxtimes n}\otimes \rho}\in\Aut(\Db(\mathrm{CY}_n))$ are non-standard. 

In our situation there is also another induced non-standard autoequivalence, namely $\widetilde{T_{E}^{\boxtimes n}}\colon \Db(\mathrm{CY}_n)\to \Db(\mathrm{CY}_n)$; see \cite{Ploog-Sosna}. It is given as a lift of $T_E^{\boxtimes n}=\FM_{\kp^{\boxtimes n}}\in \Aut(\Db(X^{[n]}))$, where $\kp\in\Db(X\times X)$ denotes the Fourier--Mukai kernel of $T_E$; see \cite{Ploog}.

Recall that the isomorphism classes of irreducible representations of $\sym_n$ are in bijection to the set $P(n)$ of partitions of $n$. For an exceptional object $E\in \Db(X)$ we set 
\begin{align*}&G_E:=\langle[1], T_E^{\boxtimes n},T_\rho:\,\rho\in P(n)\rangle\subset\Aut(\Db(X^{[n]})), \\
& \tilde G_E:=\langle[1], \widetilde{T_E^{\boxtimes n}},\tilde T_\rho:\,\rho\in P(n)\rangle\subset\Aut(\Db(\mathrm{CY}_n)).
\end{align*}

\begin{proposition}\label{comparisontoboxes}
Assume that $E\in \Db(X)$ is exceptional and $0\neq F\in\langle E, E\otimes \omega_X\rangle ^\perp$. We have $G_E\cong \IZ^{p(n)}\cong \tilde G_E$, where $p(n)=|P(n)|$. Furthermore, $T'\notin G_E$ and $\tilde T\notin \tilde G_E$, where $T'$ is a descent of $\tilde{T}$.
\end{proposition}

\begin{proof}
For $0\le k\le n$ we consider the objects
\[E^k\cdot F^{n-k}:=\Inf_{\sym_k\times \sym_{n-k}}^{\sym_n}(E^{\boxtimes k}\boxtimes E^{\boxtimes n})\in \Db_{\sym_n}(X^n)\,.\]
We have $T_E(E)=E[-1]$ and $T_E(F)=F$. It follows that
\begin{align}\label{1}
T_E^{\boxtimes n}\colon E^k\cdot F^{n-k}\mapsto E^k\cdot F^{n-k}[-k]\,,\quad E^{\boxtimes n}\otimes \rho\mapsto E^{\boxtimes n}\otimes \rho[-n].
\end{align}
By Remark \ref{twistrel}, the latter shows that $T_E^{\boxtimes n}$ commutes with all the $T_\rho$. 
Note that for $k<n$ and $\rho\not\simeq \rho'$ we have
\[\Ext^*(E^{\boxtimes n}\otimes \rho,E^k\cdot F^{n-k})=0=\Ext^*(E^{\boxtimes n}\otimes \rho,E^{\boxtimes n}\otimes \rho')\,.\]
Thus, by Equations (\ref{coneofT}) and (\ref{Taction}), and Remark \ref{twistrel}
\begin{align}\label{2}T_\rho\colon &E^k\cdot F^{n-k}\mapsto E^k\cdot F^{n-k} \text{ for $k<n$}\,,\\\notag T_\rho\colon& E^{\boxtimes n}\otimes \rho'\mapsto 
\begin{cases}
 E^{\boxtimes n}\otimes \rho' &\text{ if $\rho\neq\rho'$}\\
\omega_{X^n}\otimes E^{\boxtimes n}\otimes \rho'[-(2n-1)] &\text{ if $\rho=\rho'$}\\
\end{cases}
\end{align}
which, in particular, shows that the $T_\rho$ pairwise commute. 

Now consider $\Psi= (T_E^{\boxtimes n})^a\circ \prod_{\rho} T_\rho^{b_\rho}[c]\in G_E$ and assume that $\Psi\cong \id$.
We have $\Psi(F^{\boxtimes n})=F^{\boxtimes n}[c]$ and thus $c=0$. It follows that $\Psi(E^1\cdot F^{n-1})=E^1\cdot F^{n-1}[-a]$ and thus $a=0$. Finally, $\Psi(E^{\boxtimes n}\otimes \rho)=\omega_{X^n}\otimes E^{\boxtimes n}\otimes \rho[-b_\rho(2n-1)]$ shows that $b_\rho=0$ for all $\rho$. 
The assertion that $T'\notin G_E$ follows in a similar way using Lemma \ref{notwist}. 

The identities (\ref{1}) and (\ref{2}) lift to identities in $\Db(\mathrm{CY}_n)$. Therefore, one can analogously show that $\tilde G_E\cong \IZ^{p(d)}$ and that $\tilde T\notin G_E$.   
\end{proof}

\begin{remark}
Let $\alt$ be the alternating representation, i.e. the one-dimensional representation on which $\sym_n$ acts by multiplication by $\sgn$. By Remark \ref{twistrel} we have $T_\alt\cong \mathrm{M}_\alt\circ T_{E^{\boxtimes n}}\circ \mathrm{M}_\alt$ where $\mathrm{M}_\alt\in \Aut(\Db_{\sym_n}(X^n))$ is the involution $(-)\otimes \alt$. Note that for higher dimensional irreducible representations $\rho$ there is no such relation since $\mathrm{M}_\rho$ is not an equivalence.  
\end{remark}

\section{Exceptional sequences on $X^{[n]}$}\label{exceptionalsequences}

Let $Z$ be a smooth projective variety and $n\ge 2$. In this section we will construct exceptional sequences in the equivariant derived category $\Db_{\sym_n}(Z^n)$ out of exceptional sequences in $\Db(Z)$.

\begin{proposition}[{\cite[Cor.\ 1]{Sam}}]\label{Sam}
Let $Z$ and $Z'$ be smooth projective varieties with full exceptional sequences $E_1,\dots,E_k$ and $F_1,\dots, F_\ell$ respectively. Then 
\[E_1\boxtimes F_1, E_1\boxtimes F_2,\dots ,E_1\boxtimes F_\ell, E_2\boxtimes F_1,\dots,E_k\boxtimes F_\ell\]
is a full exceptional sequence of $\Db(Z\times Z')$.
\end{proposition}

Let now $E_1,\dots,E_k$ be an exceptional sequence on $Z$. We consider for every multi-index $\alpha=(\alpha_1,\dots,\alpha_n)\in [k]^n:=\{1,\dots,k\}^n$ the object
\[E(\alpha):=E_{\alpha_1}\boxtimes \dots \boxtimes E_{\alpha_n}\in \Db(Z^n).\]
By Proposition \ref{Sam}, these objects form a full exceptional sequence of $\Db(Z^n)$ when considering them with the ordering given by the lexicographical order $<_{\text{lex}}$ on $[k]^n$. 
\begin{remark}\label{ext}
Let $\alpha,\beta\in [k]^n$. By the K\"unneth formula 
\begin{equation}
\label{ex}\Ext^*(E(\alpha),E(\beta))=\Ext^*(E_{\alpha_1},E_{\beta_1})\otimes \dots\otimes \Ext^*(E_{\alpha_n},E_{\beta_n}).
\end{equation}
Thus, we have $\Ext^*(E(\alpha),E(\beta))=0$, whenever $\alpha_i>\beta_i$ for an $i\in [n]$. 
\end{remark}

\begin{thm}[{\cite{Ela}}]\label{Ela2}
Let $G$ be a finite group acting on a variety $M$. Consider a (full) exceptional sequence of $\Db(M)$ of the form 
\[E_1^{(1)},\dots,E^{(1)}_{k_1},E^{(2)}_{1},\dots,E^{(2)}_{k_2},\dots, E^{(\ell)}_{1},\dots,E^{(\ell)}_{k_\ell}\]
such that $G$ acts transitively on every block $E^{(i)}_1,\dots,E^{(i)}_{k_i}$, i.e.\ for every $i\in [\ell]$ and every 
pair $a,b\in [k_i]$ there is a $g\in G$ such that $g^*E^{(i)}_a\cong E^{(i)}_b$ (and conversely for every $g\in G$ and every $a\in[k_i]$ there is a $b\in [k_i]$ such that $g^*E^{(i)}_a\cong g^*E^{(i)}_b$). Let $H_i:=\Stab_G(E^{(i)}_1)$ and assume that $E^{(i)}_1$ carries an $H_i$-linearisation, i.e. there exists an $\ke^{(i)}\in \Db_{H_i}(M)$ such that $\Res(\ke^{(i)})=E^{(i)}_1$.
Then 
\begin{align*}
&\Inf_{H_1}^G(\ke^{(1)}\otimes V^{(1)}_1),\dots,\Inf_{H_1}^G(\ke^{(1)}\otimes V^{(1)}_{m_1}),\dots,\\
&\Inf_{H_\ell}^G(\ke^{(\ell)}\otimes V^{(\ell)}_1),\dots,\Inf_{H_\ell}^G(\ke^{(\ell)}\otimes V^{(\ell)}_{m_\ell})
\end{align*}
is a (full) exceptional sequence of $\Db_G(M)$ with $V_1^{(i)},\dots, V_{m_i}^{(i)}$ being all the irreducible representations of $H_i$.
\end{thm}

\begin{proof}
In \cite{Ela} the Theorem is only stated in the case of full exceptional sequences. But one can easily infer from the proof that
non-full exceptional sequences also induce non-full exceptional sequences.   
\end{proof}

In order to apply Theorem \ref{Ela2} we have to reorder the sequence consisting of the $E(\alpha)$ as follows.
For a multi-index $\alpha\in[k]^n$ we denote the unique non-decreasing representative of its $\sym_n$-orbit by $\nd(\alpha)$.
Then we define a total order $\lhd$ of $[k]^n$ by
\[\alpha\lhd \beta:\iff \begin{cases}
                         \nd(\alpha)\lex \nd(\beta) \quad\text{or}\\
                         \nd(\alpha) =  \nd(\beta)\text{ and } \alpha\lex \beta
                        \end{cases}
\]
Now the group $\sym_n$ acts transitively on the blocks consisting of all $E(\alpha)$ with fixed $\nd(\alpha)$ because of 
$\sigma^*E(\alpha)\cong E(\sigma^{-1}\cdot \alpha)$. Furthermore, every $E(\alpha)$ has a canonical $\Stab(\alpha)$-linearisation given by permutation of the factors in the box product. It remains to show that $(E(\alpha))_\alpha$ with the ordering given by $\lhd$ is still an exceptional sequence. This follows by Remark \ref{ext} and the last item of the following lemma.
\begin{lemma}\label{combi}
Let $\alpha,\beta\in[k]^n$.
\begin{enumerate}
\item Let $\nd(\alpha)=\nd(\beta)$ but $\alpha\neq \beta$. Then there exists an $i\in [n]$ such that $\alpha_i<\beta_i$.
\item Let $\sigma\in \sym_n$. Then there exists an $i\in [n]$ such that $\alpha_i<\beta_i$ if and only if there exists a 
$j\in[n]$ such that $(\sigma\cdot\alpha)_j<(\sigma\cdot \beta)_j$. 
\item If $\nd(\alpha)\lex\nd(\beta)$, then there exists an $i\in [n]$ such that $\alpha_i<\beta_i$.
\item Let $\alpha\lhd \beta$. Then there exists an $i\in [n]$ such that $\alpha_i<\beta_i$.
\end{enumerate}
\end{lemma}

\begin{proof}
If $\nd(\alpha)=\nd(\beta)$, we have $\alpha_1+\dots+\alpha_n=\beta_1+\dots+\beta_n$. This shows (1). By setting $j=\sigma(i)$ we obtain (2). In order to show (3) we may now assume using (2) that $\alpha=\nd(\alpha)$. Let $\sigma\in \sym_n$ be such that $\beta=\sigma\cdot\nd(\beta)$ and let $m:=\min\{\ell\in[n]\mid \nd(\alpha)_\ell\ne \nd(\beta)_\ell\}$. Then $\nd(\alpha)_m<\nd(\beta)_m$. 
If $\sigma(m)\le m$, we have $\alpha_{\sigma(m)}=\nd(\alpha)_{\sigma(m)}\le \nd(\alpha)_m<\nd(\beta)_m=\beta_{\sigma(m)}$. If $\sigma(m)>m$, there exists $\ell>m$ such that $\sigma(\ell)\le m$. This yields \[\alpha_{\sigma(\ell)}=\nd(\alpha)_{\sigma(\ell)}\le\nd(\alpha)_m<\nd(\beta)_m\le\nd(\beta)_\ell=\beta_{\sigma(\ell)}.\]
Finally, (4) follows from (1) and (3).    
\end{proof}

We summarize all of the above in the following

\begin{proposition}
If $\alpha\in[k]^n$ is a non-decreasing multi-index and $V_i^{(\alpha)}$ is an irreducible representation of $H_\alpha=\Stab(\alpha)$, then the collection of objects $\ke(\alpha,V_i^{(\alpha)}):=\Inf_{H_\alpha}^{\sym_n}\left(E(\alpha)\otimes V^{(\alpha)}_i\right)$ forms an exceptional sequence of $\Db_{\sym_n}(Z^n)$. The induced sequence is full if and only if the original sequence on $\Db(Z)$ is full.\qqed 
%The collection of objects $\ke(\alpha,V_i^{(\alpha)}):=\Inf_{H_\alpha}^{\sym_n}\left(E(\alpha)\otimes V^{(\alpha)}_i\right)$ for
%$\alpha\in[k]^n$ a non-decreasing multi-index and $V_i^{(\alpha)}$ an irreducible representation of $H_\alpha=\Stab(\alpha)$ 
%forms an exceptional sequence of $\Db_{\sym_n}(Z^n)$. The induced sequence is full if and only if the original sequence on $\Db(Z)$ is full.\qqed 
\end{proposition}

\begin{remark}
 An exceptional sequence is called \textit{strong} if all the higher extension groups between its members vanish.
Using Equation (\ref{ex}) one can show that $(\ke(\alpha,V^{(\alpha)}_i))_{\alpha,i}$ is strong if and only if 
 $(E_\ell)_\ell$ is strong. Thus, in the case that the full exceptional sequence $E_1,\dots,E_k$ of $\Db(Z)$ is strong, there is an equivalence of triangulated categories $\Db_{\sym_n}(Z^n)\simeq \Db(\text{Mod}-\text{End}_{\sym_n}(\km))$
where $\km:=\bigoplus_{\alpha,i} \ke(\alpha,V^{(\alpha)}_i)$; see \cite{Bon}.
\end{remark}

\begin{remark}
Using \cite{Ela2} one can also construct semi-orthogonal decompositions of $\Db_{\sym_n}(Z^n)$ out of general semi-orthogonal decompositions of $\Db(Z)$ in a similar way. Furthermore, if $T\in \Db(Z)$ is a tilting object, so is $\oplus_{\rho\in P(n)}(T^{\boxtimes n}\otimes \rho)\in \Db_{\sym_n}(Z^n)$.   
\end{remark}

\begin{remark}
Any Enriques surface has a completely orthogonal exceptional sequence $(E_i)_i$ of length $10$; see \cite{Zube}. The induced exceptional sequence in $\Db_{\sym_n}(X^n)\cong \Db(X^{[n]})$ is again completely orthogonal and the same holds for the corresponding sequence of spherical objects on $CY_n$. By \cite[Cor.\ 2.4]{Krug1} it follows that the spherical twists give an embedding $\IZ^{\ell(10,n)}\hookrightarrow \Aut(\Db(\mathrm{CY}_n))$ where $\ell(10,n)$ denotes the length of the induced sequence.
By arguments similar to those in the proof of Proposition \ref{comparisontoboxes} we also get an embedding $\IZ^{\ell(10,n)}\hookrightarrow \Aut(\Db(X^{[n]}))$.
\end{remark}

\begin{remark}
Let $Z=X$ be an Enriques surface and let \[\tilde \ke(\alpha):=\pi^*(\ke(\alpha,\triv))\in \Db(\mathrm{CY}_n).\] In the case that $\tilde E_1,\dots,\tilde E_k$ is an $A_k$-sequence of spherical objects on $\Db(\tilde X)$, the sequence
\[\tilde \ke(1,\dots,1),\tilde\ke(1,\dots,2),\dots,\tilde \ke(2,\dots,2),\dots,\tilde \ke(k,\dots,k)\]
is a $A_{(n-1)k+1}$ sequence of spherical objects in $\Db(\mathrm{CY}_n)$. Thus, there is an embedding 
$B_{(n-1)k+1}\hookrightarrow \Aut(\Db(\mathrm{CY}_n))$ of the braid group; see \cite{Seidel-Thomas}. 
\end{remark}
\section{The truncated universal ideal functor}\label{truncated}

The arguments in this section follow those of \cite[Sect.\ 6]{Meachan}. The key new observation is that the functor $\hat F:=\Phi F\colon \Db(Z)\to \Db_{\sym_n}(Z^n)$ for $Z=S$ a surface can be truncated to a functor $G$ in such a way that $G^RG\cong \hat F^R\hat F$ and that this functor generalises in a nicer way to varieties of arbitrary dimension than $\hat F$ does.    

If $Z=S$ is a surface, then $\hat F''=\Phi\FM_{\ko_{\kz_n}}= \FM_{\mathcal C^\bullet}$, where $\mathcal C^\bullet$ is the complex concentrated in degrees $0,\dots,n-1$ given by
\[0\to \bigoplus\limits_{i=0}^n\reg_{D_i}\to \bigoplus\limits_{|I|=2}\reg_{D_I}\otimes \alt_I\to \bigoplus\limits_{|I|=3}\ko_{D_I}\otimes \alt_I \dots \to \reg_{D_{[n]}}\otimes \alt_{[n]}\to 0\,;\]
see \cite{Sca1}.
For $I\subset [n]:=\{1,\dots,n\}$, the reduced subvariety $D_I\subset S\times S^n$ is given by 
$D_I=\{(y,x_1,\dots,x_n)\mid y=x_i\,\forall\, i\in I\}$ and $\alt_I$ is the alternating representation of $\sym_I$. 
Furthermore, $\hat F'=\Phi \FM_{\reg_{S\times S^{[n]}}}\cong \FM_{\reg_{S\times S^n}}$ and the induced map $\hat F'\to \hat F''$ is given by the morphism of kernels $\reg_{S\times S^n}\to \mathcal C^0=\oplus_i\reg_{D_i}$ whose components are given by restriction of sections. We set $G'':=\FM_{\mathcal C^0}$.
As explained in \cite[Sect.\ 6]{Meachan}, the main steps in the computation of the formulas of \cite{Sca1} and \cite{Krug} can be translated into the following statement
\begin{align}\label{cbul}\hat F'^R\hat F''\cong \hat F'^RG'',\quad \hat F''^R F'\cong G''^RF',\quad \hat F''^R \hat F''\cong G''^RG''.\end{align}

\begin{definition}
Let $Z$ be a smooth projective variety of arbitrary dimension $d$ and $n\ge 2$. 
The \textit{truncated universal ideal functor}
$G=\FM_\G\colon \Db(Z)\to \Db_{\sym_n}(Z^n)$
is the Fourier--Mukai transform whose kernel is the complex
\[\G:=\G^\bullet:=(0\to \reg_{Z\times Z^n}\to \oplus_{i=1}^n\reg_{D_i}\to 0)\in \Db_{\sym_n}(Z\times Z^n).\]
\end{definition}

Thus, there is the triangle of FM transforms $G\to G'\to G''$ with
\[
 G':=\FM_{\G'}=\hat F'=\Ho^*(X,-)\otimes \reg_{X^n},\quad G''=\FM_{\G''}=\Inf_{\sym_{n-1}}^{\sym_n} p_{n}^* \triv.
\]
For $E\in \Db(Z)$ we have $G''(E)=\oplus_{i=1}^n p_i^*E$ (see also Subsection \ref{subsection:equivariant} for details about the inflation functor $\Inf$ and its right adjoint $\Res$). 
The right-adjoints are 
\begin{align*}
 G'^R=\FM_{\G'^R}=\Ho^*(Z^n,-)^{\sym_n}\otimes \omega_Z[d]&,\quad \G'^R=\reg_{Z^n}\boxtimes\omega_Z [d],\\
 G''^R=\FM_{\G''^R}=[-]^{\sym_{n-1}}\circ p_{n*}\circ \Res_{\sym_n}^{\sym_{n-1}}&,\quad \G''^R=\oplus_{i=0}^n \reg_{D_i}.\end{align*}
In the surface case Equation (\ref{cbul}) gives the following
\begin{lemma}\label{surftrunc}
If $Z=S$ is a surface, $\hat F^R\hat F\cong G^RG$.\qqed
\end{lemma}

We compute the compositions of the kernels:
\begin{align*}
\G'^R\G'&=(\reg_Z\boxtimes \omega_Z)\otimes S^n\Ho^*(\reg_Z)[d],\\
\G'^R\G''&=(\reg_Z\boxtimes \omega_Z)\otimes S^{n-1}\Ho^*(\reg_Z)[d],\\
\G''^R\G'&=(\reg_Z\boxtimes \reg_Z)\otimes S^{n-1}\Ho^*(\reg_Z),\\
\G''^R\G''&=\reg_\Delta\otimes S^{n-1}\Ho^*(\reg_Z)\oplus (\reg_Z\boxtimes \reg_Z)\otimes S^{n-2}\Ho^*(\reg_Z).
\end{align*}
The induced map $\G'^R\G'\to \G'^R\G''$ under these isomorphisms is given by evaluation as follows. Let $(e_i)_i$ be a basis of $\Ho^*(\reg_Z)=\Hom(\reg_Z,\reg_Z[*])$. Then the component \[(\reg_Z\boxtimes \omega_Z)\cdot e_{i_1}\cdots e_{i_n}[d]\to (\reg_Z\boxtimes \omega_Z)\cdot e_{i_1}\cdots\hat e_{i_k}\cdots e_{i_n}[d]\] is $e_{i_k}\boxtimes \id[d]$.  
The component 
\[(\reg_Z\boxtimes \reg_Z)\otimes S^{n-1}\Ho^*(\reg_Z)\to (\reg_Z\boxtimes \reg_Z)\otimes S^{n-2}\Ho^*(\reg_Z)\] 
of $\G''^R\G'\to \G''^R\G''$ is given in the same way and the component 
\[(\reg_Z\boxtimes \reg_Z)\otimes S^{n-1}\Ho^*(\reg_Z)\to \reg_\Delta\otimes S^{n-1}\Ho^*(\reg_Z)\]
is the restriction map. The map $\G''^R\G'\to \G'^R\G'$ is given as follows. Let $(\theta_i)_i$ be the basis of $\Ho^*(\reg_Z)^\vee$ dual to $(e_i)_i$. Furthermore, let $\theta_i$ correspond to the morphism 
$\tilde \theta_i\in \Hom(\reg_Z,\omega_Z[d-*])\cong \Hom(\reg_Z,\reg_Z[*])^\vee$ under Serre duality. Then the component
\[(\reg_Z\boxtimes \reg_Z)\cdot e_{i_1}\cdots\hat e_{i_k}\cdots e_{i_n}\to (\reg_Z\boxtimes \omega_Z)\cdot e_{i_1}\cdots e_{i_n}[d]\]
is $\tilde \theta_{i_k}$. 

In the following we will use the commutative diagram
 \begin{align}\label{lattice}\xymatrix{
          \G''^R\G\ar[r]\ar[d]  & \G''^R\G' \ar[r]\ar[d]  & \G''^R\G'' \ar[d]\\
  \G'^R\G \ar[r]\ar[d]   & \G'^R\G' \ar[r]\ar[d] & \G'^R\G''   \ar[d]  \\
    \G^R\G \ar[r]   &  \G^R\G' \ar[r]   & \G^R\G'' \,.
} \end{align}
with exact triangles as columns and rows in order to deduce formulae for $G^RG=\FM_{\G^R\G}$.

\subsection{The case of an even dimensional Calabi--Yau variety}
If $Z$ is a Calabi--Yau variety of even dimension $d$, then $\Ho^*(\reg_Z)=\IC[0]\oplus \IC[-d]$ and $S^k\Ho^*(\reg_Z)=\IC[0]\oplus \IC[-d]\oplus\dots\oplus \IC[-dk]$. Let $u\in \Ho^d(\reg_Z)$ be the basis vector whose dual
$\theta\in \Hom(\reg_Z,\reg_Z[d])^\vee$ corresponds to $\id\in \Hom(\reg_Z,\reg_Z)$ under Serre duality. We denote the induced degree $d\ell$ basis vector of $S^k\Ho^*(\reg_Z)$ by $u^\ell$.
 
\begin{lemma}
For a Calabi--Yau variety $Z$ of even dimension $d$ we have \[G^RG\cong \id\oplus [-d]\oplus\dots \oplus [-d(n-1)]\,.\]  
\end{lemma}
\begin{proof}
By the description of the last subsection, the components $\reg_{Z^2}\cdot u^\ell\to \reg_{Z^2}\cdot u^{\ell}$ of $\G''^R\G'\to \G''^R\G''$ and $\G'^R\G'\to \G'^R\G''$ equal the identity. By \cite[Lem.\ 5]{Hub}, the top of diagram (\ref{lattice}) is isomorphic to
\begin{align*}\xymatrix{
          \G''^R\G \ar[r]\ar[d] & \reg_{Z^2}[-d(n-1)] \ar[r]\ar[d]  & \reg_\Delta([0]\oplus[-d]\oplus\dots\oplus [-d(n-1)]) \ar[d]\\
  \G'^R\G \ar[r]   & \reg_{Z^2}[-d(n-1)] \ar[r] & 0  
} 
\end{align*}
where the middle vertical map is the component $\reg_{Z^2} \cdot u^{n-1}\to \reg_{Z^2} u^n[d]$ of $\G''^R\G'\to \G'^R\G'$. By the above description it is the identity. Now the claim follows by the octahedral axiom.
\end{proof}
\begin{thm}
If $Z$ is a Calabi--Yau variety of even dimension, then $G\colon \Db(Z)\to \Db_{\sym_n}(Z^n)$ is a $\IP^{n-1}$-functor.
\end{thm}
\begin{proof}
 By the previous lemma, condition (1) of the definition of a $\IP^{n-1}$-functor is satisfied.

The proof that condition (2) is satisfied is analogous to the proof in the case of the non-truncated functor when $X$ is a K3 surface. Basically, one has to go through \cite[Sec.\ 2.5]{Addington} and replace $F$ by $G$, $q^*$ by $p_n^*\circ \triv$ and its right adjoint $q_*$ by $(-)^{\sym_{n-1}}\circ p_{n*}$, $g_*$ by $\Inf$ and its right adjoint $g^!$ by $\Res$, $\HH^*(S^{[n]})$ by $\HH^*(Z^n)^{\sym_n}$, $\Ho^*(\reg_{S^{[n]}})$ by $\Ho^*(Z^n)^{\sym_n}$, 2 by $d$, and the symplectic form $\sigma$ by a generator of $\Ho^d(\reg_Z)$.

Very roughly, the idea is the following: $G^RG$ is identified with the direct summand $(p_n^*\circ \triv)^R(p_n^*\circ \triv)\cong \id \otimes \Ho^*(Z^{n-1},\reg_{Z^{n-1}})^{\sym_{n-1}}$ of $G''^RG''$ and the monad structure of $(p_n^*\circ \triv)^R(p_n^*\circ \triv)$ is given by the cup product on $\Ho^*(Z^{n-1},\reg_{Z^{n-1}})^{\sym_{n-1}}$ which has the right form.

Condition (3) is easy to check since all occurring autoequivalences are simply shifts. 
\end{proof}

\subsection{The case of an odd dimensional Calabi--Yau variety}
In this case $\Ho^*(\reg_Z)=\IC[0]\oplus \IC[-d]$ and $S^k\Ho^*(\reg_Z)=\IC[0]\oplus \IC[-d]$ for $k\ge 1$. The reason for the vanishing of the higher degrees is that the symmetric product is taken in the graded sense.
\begin{lemma}
 For $n\ge 3$ we have $\G^R\G\cong \reg_\Delta([0]\oplus [-d])$.
\end{lemma}
\begin{proof}
 The component $\reg_{Z^2} \otimes S^{n-1}\Ho^*(\reg_Z)\to \reg_{Z^2} \otimes S^{n-2}\Ho^*(\reg_Z)$ of the map $\G''^R\G'\to \G''^R\G''$ as well as the whole $\G'^R\G'\to \G'^R\G''$ are isomorphisms. This yields $\G''^R\G\cong \reg_\Delta\otimes S^{n-1}\Ho^*(\reg_Z)[-1]\cong \reg_{\Delta}([0]\oplus[-d])[-1]$ and $\G'^R\G=0$. The result now follows from the triangle $\G''^R\G\to \G'^R\G\to \G^R\G$.
\end{proof}
That means that the cotwist of $G$ is an equivalence. For dimension reasons the second axiom of a spherical functor cannot hold for $G$. Thus, one could call $G$ a \textit{sphere-like functor} in analogy with the sphere-like objects of 
\cite{HKP}.
\subsection{The case $\Ho^*(\reg_Z)=\IC[0]$}
In this case also $S^k\Ho^*(\reg_Z)=\IC[0]$ for $k\ge 1$.
\begin{proposition}
The functor $G$ is fully faithful, i.e. $\G^R\G=\reg_\Delta$.
\end{proposition}
\begin{proof}
The maps $\G''^R\G'\to \G''^R\G''$ and $\G'^R\G'\to \G'^R\G''$ are the identity on the components $\reg_{Z^2}$. Hence, $\G''^R\G\cong \reg_\Delta\otimes S^{n-1}\Ho^*(\reg_Z)[-1]$ and $\G'^R\G=0$. It follows that $\G^R\G=\reg_\Delta\otimes S^{n-1}\Ho^*(\reg_Z)=\reg_\Delta$. 
\end{proof}
\begin{remark}
In particular, when $Z=S$ is a surface, the Proposition in conjunction with Lemma \ref{surftrunc} reproves Theorem \ref{hilb-semiorth}. 
\end{remark}
\begin{remark}
In contrast to the case of the non-truncated ideal functor (see Remark \ref{kerR}) it is, at least for $n\ge 3$, easy to find an object in $\ker G^R$, namely $\reg_{Z^n}\otimes \alt$. Since for even $d=\dim Z$ also $G^R(\omega_{Z^n}\otimes \alt)=0$, we have in the case that $Z=X$ is an Enriques surface $T_G(\reg_{X^n}\otimes \alt)=\reg_{X^n}\otimes \alt$ by Equation (\ref{coneofT}). This shows that, although the functors $\hat F$ and $G$ are very similar, the induced twists differ at least slightly. It also follows by Remark \ref{twistrel} that $T_G$ and $T_{\reg_{X^n}\otimes \alt}$ commute.  
\end{remark}

\end{document}